\documentclass{article}

\usepackage[english]{babel}
\usepackage[utf8x]{inputenc}
\usepackage[T1]{fontenc}

\usepackage[a4paper,top=3cm,bottom=3cm,left=3cm,right=3cm,marginparwidth=1.75cm]{geometry}

\usepackage{amsmath}
\usepackage{amsthm}
\usepackage{amssymb}
\usepackage{graphicx}
\usepackage{wrapfig}
\usepackage[colorinlistoftodos]{todonotes}
\usepackage[colorlinks=true, allcolors=blue]{hyperref}
\usepackage{mathtools}

\newtheorem{theorem}{Theorem}
\newtheorem{lemma}[theorem]{Lemma}
\newtheorem{conj}[theorem]{Conjecture}

\theoremstyle{definition}
\newtheorem*{definition}{Definition}

\title{An upper bound for the restricted online Ramsey number}
\author{David Gonzalez\thanks{Corresponding Author. Department of Mathematics, Stanford University, Stanford, CA 94305, USA. Email: dgonz7@stanford.edu}, Xiaoyu He\thanks{Department of Mathematics, Stanford University, Stanford, CA 94305, USA. Email: alkjash@stanford.edu. Research supported by a NSF GRFP grant number DGE-1656518.}, Hanzhi Zheng\thanks{ Department of Mathematics, Stanford University, Stanford, CA 94305, USA. Email: zhz108@stanford.edu}}

\begin{document}
\maketitle

\begin{abstract}
The restricted $(m,n;N)$-online Ramsey game is a game played between two players, Builder and Painter. The game starts with $N$ isolated vertices. Each turn Builder picks an edge to build and Painter chooses whether that edge is red or blue, and Builder aims to create a red $K_m$ or blue $K_n$ in as few turns as possible. The restricted online Ramsey number $\tilde{r}(m,n;N)$ is the minimum number of turns that Builder needs to guarantee her win in the restricted $(m,n;N)$-online Ramsey game. We show that if $N=r(n,n)$,
\[
\tilde{r}(n,n;N)\le \binom{N}{2} - \Omega(N\log N),
\]
motivated by a question posed by Conlon, Fox, Grinshpun and He. The equivalent game played on infinitely many vertices is called the online Ramsey game. As almost all known Builder strategies in the online Ramsey game end up reducing to the restricted setting, we expect further progress on the restricted online Ramsey game to have applications in the general case.

\end{abstract}

\section{Introduction}

Ramsey's theorem states that for any $m, n\ge 3$, there exists a least positive integer $N = r(m,n)$ such that any red-blue coloring of the edges of the complete graph $K_{N}$ contains either a red $m$-clique or blue $n$-clique. These particular integers $r(m,n)$ are the \textit{Ramsey numbers}. Determining the growth rate of the Ramsey numbers $r(m,n)$ is perhaps the central problem of Ramsey theory, and much is still unknown. An early result of Erd\H os and Szekeres guarantees that $r(n,n)\leq 2^{2n}$ \cite{ErSk}, and in the other direction Erd\H os proved that $r(n,n) \ge 2^{n/2}$ \cite{Erdos}. No exponential improvement has been made on either bound in the decades since they were proven.

This paper is concerned with a widely-studied variant of Ramsey numbers, called \textit{online Ramsey numbers}. First we define the \textit{online Ramsey game}, which is played between two players called Builder and Painter. 

Fix positive integers $m, n\ge 3$. The game takes place on an infinite set of isolated vertices. Each turn, Builder chooses two non-adjacent vertices and builds the edge between them. Painter then paints the edge either red or blue. Builder wins when a red $m$-clique or blue $n$-clique appears in the graph, and Painter's goal is to prevent Builder's win for as long as possible. The \textit{online Ramsey number} $\tilde{r}(m,n)$ is the smallest $t$ such that Builder has a strategy to win within $t$ turns regardless of how painter plays.

Online Ramsey numbers were first introduced by Beck \cite{beck} and independently by Kurek and Ruciński \cite{kr}. One can easily find an exponential bound on the online Ramsey number $\tilde{r}(m,n)$ using the classical Ramsey number as follows.
\begin{equation} \label{eq:trivial}
\frac{r(m,n)}{2} \leq \tilde{r}(m,n) \leq \binom{r(m,n)}{2}
\end{equation}

However, unlike the classical Ramsey numbers which have seen no exponential improvements in decades, both sides of (\ref{eq:trivial}) have been improved. Conlon \cite{Conlon} proved an exponential improvement on the upper bound, showing that for infinitely many $n$
\[\tilde{r}(n,n) \leq 1.001^{-n}\binom{r(n,n)}{2}.\]

In the other direction, Conlon, Fox, Grinshpun, and He \cite{xiaoyu} used the probabilistic method to prove an exponential improvement to the lower bound as well, showing for all $n\ge 3$,
\[\tilde{r}(n,n) \geq 2^{(2-\sqrt{2})n + O(1)}.\]

Fix positive integers $m, n$, and $N$. The $(m,n;N)$-\textit{restricted online Ramsey game} is the online Ramsey game played on a finite vertex set of size $N$. The \textit{restricted online Ramsey number} $\tilde{r}(m,n;N)$, defined in \cite{xiaoyu}, is the number of moves Builder must take to ensure victory in this game. Of course, this number is only defined when $N \geq r(m,n)$. Many of the Builder strategies used throughout \cite{xiaoyu} and \cite{Conlon} reduce the $(m,n)$-online Ramsey game to the $(m',n';r(m',n'))$-restricted game, for some $m' < m$ and $n' < n$, and then apply the trivial bound
\[
\tilde{r}(m',n';r(m',n')) \le \binom{r(m',n')}{2}.
\]
Improved bounds on the restricted online Ramsey number may allow for corresponding improvements in these Builder strategies, motivating our work on these problems.

Our main result on the restricted online Ramsey game is the following.
\begin{theorem}\label{thm:main}
If $n\ge 3$ and $N = r(n,n)$, then
\[
\tilde{r}(n,n;N)\leq{N \choose 2} - \Omega(N\log N).
\]
\end{theorem}
Since $\binom{N}{2}$ is the total number of edges in $K_N$, we say that Builder may always save $\Omega(N\log N)$ moves in the restricted $(n,n;N)$-online Ramsey game.

The rest of this paper is organized as follows. In the next section, we make some basic definitions and collect some useful results from extremal graph theory. After that, we divide the proof of Theorem~\ref{thm:main} into two sections. Section~\ref{sec:medbip} shows that if Builder builds a large complete multipartite graph, some pair of parts will have many edges of both colors between them. Section~\ref{sec:indpair} then shows that within these two parts of the graph, Builder can save many moves by constructing a large family of what we call {\it independent pairs}.

Finally in Section~\ref{sec:closing} we mention some open problems surrounding the restricted online Ramsey number.

\section{Preliminaries}
Henceforth, we let $N=r(n,n)$, and simply call the $(n,n;N)$-restricted online Ramsey game the $(n,n;N)$-game.

A bichromatic graph is a graph whose edges are colored red or blue.
\begin{definition}
A \textit{bichromatic graph} is a triple of sets $G=(V,R,B)$ where $R,B\subseteq {V\choose 2}$ and $R\cap B=\emptyset$. We say that $V$ is the vertex set and $R$ (resp. $B$) is the set of red (resp. blue) edges of $G$.
\end{definition}

If $G$ is a bichromatic graph, write $d_R (G) = \frac{|R|}{|R+B|}$ for the density of red edges (out of all the edges) in $G$. We say that a bichromatic graph is {\it $\varepsilon$-color-balanced} if $\varepsilon \le d_R(G) \le 1-\varepsilon$. Define induced bichromatic subgraphs in the obvious way.

The \textit{underlying graph} of a bichromatic graph $G=(V,R,B)$ is the (uncolored) graph $(V,R\cup B)$ and we say an (uncolored) graph is contained in $G$ if it is a subset of the underlying graph of $G$.

A bichromatic graph is complete if its underlying graph is complete.

We will think of bichromatic graphs as intermediate states in the $(n,n;N)$-game, and show that if Builder can reach certain bichromatic bipartite graphs $G$, Builder will be able to save many moves from those $G$.

\begin{definition}
If $G$ is a bichromatic graph on $N$ vertices, define $s(G)$ to be the largest $s\in \mathbb{N}$ for which Builder can win the $(n,n;N)$-game starting from $G$ using $\binom{N}{2}-e(G)-s$ moves. 
\end{definition}

We can now restate our Theorem~\ref{thm:main} in terms of $G$.
\begin{theorem}\label{thm:nlogn}
If $G$ is the empty bichromatic graph on $N$ vertices, 
$$s(G)=\Omega(N\log N).$$
\end{theorem}

In other words, Builder can always save $\Omega(N\log N)$ edges in the $(n,n,N)$-game. In order to prove this theorem, we look to build structures from which Builder can always save moves.

For any bichromatic graph $G=(V,R,B)$, we call a pair $(u,v)\in V^2$ of non-adjacent vertices an {\it unbuilt pair}.

\begin{definition}
If $G = (V, R, B)$ is a bichromatic graph, two vertex-disjoint pairs $(u_1,u_2)$ and $(v_1,v_2)$ in $V^2$ are {\it independent} if both are unbuilt and there exist both a red edge $u_i v_j \in R$ and a blue edge $u_{i'} v_{j'} \in B$ for some (not necessarily distinct) $i,i',j,j'\in\{1,2\}$.
\end{definition}

The essential observation is that if two pairs $(u_1, u_2)$ and $(v_1, v_2)$ are independent in $G$, then the four vertices $u_1, u_2, v_1, v_2$ will never be in the same monochromatic clique. 

\begin{lemma}\label{lem:pairwise}
If $G$ is a bichromatic graph on $N$ vertices containing $s$ pairs $p_1,\ldots, p_s$ each independent to each of $t$ pairs $q_1,\ldots, q_t$, then $s(G)\ge \min(s,t)$.
\begin{proof}
Builder's strategy from this point forward is to build all the edges other than $p_1,\ldots, p_s$ and $q_1, \ldots, q_t$. Once this is done, let $G'$ be the resulting bichromatic graph. We claim that Builder only needs to build either $p_1,\ldots, p_s$ or $q_1, \ldots, q_t$ in $G'$ to win. Indeed, if all pairs $p_1,\ldots, p_s$ and $q_1, \ldots, q_t$ are built, the resulting graph is complete of size $N$ and must contain a monochromatic clique by Ramsey's theorem.

But independent pairs cannot lie in a monochromatic clique together. In particular, either building all the $p_i$ or all the $q_j$ alone will force a monochromatic clique already. The result follows.
\end{proof}
\end{lemma}

In general, if $G$ is a bichromatic graph on $N$ vertices containing unbuilt pairs $p_{j,1},\ldots, p_{j,s_j}$ for all $j=1,\cdots,t$ (so there are $\sum_{j=1}^ts_j$ of them), such that $p_{j_1,k}$ is independent to $p_{j_2,l}$ for any $j_1\neq j_2$, then $s(G)\ge s_1+s_2+\cdots+s_t-\max(s_1,s_2,\cdots,s_t)$. 

We then collect two old results in extremal graph theory that we will need. The first shows that a sparse graph contains a larger clique or independent set than Ramsey's theorem predicts.

\begin{lemma}\label{lem:bigclique} (Erd\H os and Szemer\'edi \cite{es}.)
There exists a universal constant $a>0$ such that if $G$ is a graph on $N$ vertices and $r\leq \varepsilon N^2$ edges, then either $G$ or its complement contains clique of size $s$ where 
\[
s > \frac{a\log N}{\varepsilon \log (\varepsilon^{-1})}.
\]
\end{lemma}

The second result is a theorem of K\"ov\'ari, S\'os, Tur\'an \cite{Kovari}, answering the famous problem of Zarankiewicz. It shows that dense bipartite graphs contain large complete bipartite subgraphs.
\begin{lemma} (K\"ov\'ari, S\'os, Tur\'an \cite{Kovari}.)\label{lem:Kovari}
Suppose $m\ge s\ge 1$ and $n\ge t\ge 1$, $G = (U,V,E)$ is a bipartite graph for which $|U| = m$ and $|V|=n$, and $G$ contains no subgraph isomorphic to $K_{s,t}$. Then,
\[
|E| < (t-1)^{1/s}(m-s+1)n^{1-1/s}+(s-1)n.
\]
\end{lemma}

We only need a straightforward consequence of the above result.

\begin{lemma}\label{lem:computation}
Suppose $G=(U,V,E)$ is a bipartite graph satisfying $|U|=N_0,|V|=\binom{N_0}{2}$ with at least $\delta|U||V|$ edges, and $N_0$ is sufficiently large. Then, $G$ contains a copy of $K_{a,b}$, where $a = \delta \log N_0$ and $b=N_0\log N_0$. 
\begin{proof}
Suppose otherwise. Using Lemma~\ref{lem:Kovari}, we can compute that for $N_0$ sufficiently large,
\begin{align*}
    |E|&<(N_0\log N_0-1)^{\frac{1}{\delta \log N_0}}(N_0-\delta \log N_0+1)\binom{N_0}{2}^{1-\frac{1}{\delta \log N_0}}+(\delta \log N_0-1)\binom{N_0}{2}\\
    &<(N_0\log N_0)^{\frac{1}{\delta \log N_0}}N_0(N_0^2)^{1-\frac{1}{\delta \log N_0}}+\delta \log N_0\cdot N_0^2\\
    &<2^{-\frac{1}{\delta}}(\log N_0)^{\frac{1}{\delta \log N_0}}N_0^3+o(\delta N_0^3)\\
    &<\delta N_0^3.
\end{align*}

This is a contradiction.
\end{proof}
\end{lemma}

Finally, for the sake of notational convenience in the following sections, we further make the following definitions. 

Write $K_{(X \times Y)}$ for the complete multipartite graph with $X$ parts $U_1, \ldots, U_X$ each of size $Y$. If $G = (V,R,B)$ has underlying graph $K_{(X\times Y)}$, define the {\it $\varepsilon$-reduced graph} of $G$ to be the graph $G_{\varepsilon}$ whose vertices are the parts $U_i$ of $G$, and whose edges are defined as follows. If $U_i, U_j$ are distinct parts of $G$, then there is a red edge between them in $G_{\varepsilon}$ if $d_R(G[U_i \cup U_j]) > 1 - \varepsilon$, and a blue edge if instead $d_R(G[U_i \cup U_j]) < \varepsilon$.

\section{Constructing a color-balanced bipartite graph}\label{sec:medbip}

The first step of the Builder strategy is to construct a bichromatic graph $G$ with a large color-balanced bipartite subgraph.

The main lemma of this section is that if Builder starts by building a $K_{(X\times Y)}$ and the reduced graph $G_\varepsilon$ turns out to be complete, then Builder can save many moves in the $(n,n;N)$-game.
\begin{lemma}\label{thm:multipartite}
For all $n$, there exist universal constants $C, \varepsilon > 0$ such that if $G$ is a bichromatic graph on $N = R(n,n)$ vertices with induced subgraph $K_{(C \times N/C)}$ and the $\varepsilon$-reduced graph $G_{\varepsilon}$ of $G$ is complete, then
\[
s(G) \ge \Omega(N^2).
\]
\end{lemma}

\begin{proof}
Suppose we have a bichromatic graph $G=(V,R,B)$ with a complete $\varepsilon$-reduced graph $G_{\varepsilon} = (V',R',B')$. We can apply the Erd\H os-Szekeres upper bound for $r(n,n)$ to see that 
\[
|V'|=C>r(1/2\log_2(C),1/2\log_2(C)).
\]

Therefore, there is a monochromatic clique of size $t = 1/2\log_2(C)$ in $G_{\varepsilon}$. Without loss of generality, let it be blue, and let its vertices be $U_1, \ldots, U_t$.

The definition of the reduced graph tells us that $H = G[U_1\cup \ldots \cup U_t]$ is complete multipartite with $t$ parts and between each pair of distinct parts at most an $\varepsilon$ fraction of the edges are blue. Therefore overall, $d_R(H) < \varepsilon$. 

It suffices to show that Builder can always guarantee a monochromatic clique of size $n$ by building all the unbuilt pairs in $H$, since then Builder wins with $\Omega(N^2)$ pairs unbuilt in the rest of $G$.

First, note that the only pairs unbuilt in $H$ are the pairs within individual parts, of which there are $t\binom{N/C}{2}$. For $C$ sufficiently large, this is less than an $\varepsilon$ fraction of all the pairs of vertices in $H$. In particular, since $d_R(H) < \varepsilon$, if Builder builds the remaining edges within $H$, the resulting complete bichromatic graph $H'$ satisfies $d_R(H') < 2 \varepsilon$ regardless of Painter's choices. From Lemma~\ref{lem:bigclique} it directly follows that, regardless of Painter's choices, $H'$ will contain a monochromatic clique of size at least
\[
u = \Omega \Big( \frac{\log (N)}{\varepsilon \log (\varepsilon^{-1})}\Big).
\]

Since $n = \Theta( \log N)$, taking $\varepsilon$ small enough and $C$ large enough in terms of $\varepsilon$, we can make $u > n$. 

In total, we have exhibited a strategy which guarantees a monochromatic $n$-clique if we begin with  a complete multipartite graph $K_{(C\times N/C)}$ with some bipartite part not $\varepsilon$-balanced. Given this strating point, we can then fill in the unbuilt pairs in $O(\log C)$ of the parts. Within the remaining $(1-o(1))C$ parts Builder thus saves $\Omega(N^2)$ moves, completing the proof.
\end{proof}

Lemma~\ref{thm:multipartite} handles the case when $G_\varepsilon$ is complete. The remaining case is that for some pair of parts $U_i, U_j$, the induced subgraph $G[U_i \cup U_j]$ is color-balanced. We handle this case in the next section. 

\section{Independent pairs}\label{sec:indpair}

By Lemma~\ref{thm:multipartite}, it will suffice to solve the problem for color-balanced bipartite graphs $G$. To do so we show that such graphs contain many independent pairs.

\begin{definition}
Let $G = (V, R, B)$ be a bipartite bichromatic graph with bipartition $V = V_1 \sqcup V_2$. The \textit{left vertex-pair incidence graph} of $G$ is the bipartite graph $H_L = (V', E')$ with vertex set $V' = V_1 \sqcup \binom{V_2}{2}$ and $(u,(v_1,v_2))\in E'$ if and only if $(u,v_1)$ and $(u,v_2)$ are two edges in $G$ of different colors.
\end{definition}

Observe that if $(u, (v_1, v_2))$ is an edge of the left vertex-pair incidence graph of $G$, then every unbuilt pair containing $u$ is independent from $(v_1, v_2)$.

Define the \textit{right vertex-pair incidence graph} $H_R$ of $G$ to be the left vertex-pair incidence graph of $G$ with the sides of its bipartition $V_1 \sqcup V_2$ swapped.

\begin{lemma}\label{lem:leastdensity}
For all $\varepsilon >0$, there exists $\delta>0$ depending only on $\varepsilon$ such that if $G$ has underlying graph $K_{N_0, N_0}$ for sufficiently large $N_0$ and $G$ is $\varepsilon$-color-balanced, then at least one of its vertex-pair incidence graphs $H_L$ and $H_R$ has at least $\delta N_0^3$ edges.
\begin{proof}
Suppose  $G$ has vertex bipartition $V=V_1 \sqcup V_2$. Let $\delta=\frac{\varepsilon^5}{2(1+\varepsilon)}$, $\mu=\varepsilon^2$, and $\nu=\frac{\varepsilon}{2(1+\varepsilon)}$. 

Define a vertex in $V_1$ to be color-balanced if it has at least $\mu N_0$ neighbors of each color. Then, for each color-balanced vertex in $V_1$, its corresponding vertex in $H_L$ has degree at least $(\mu N_0)^2=\mu^2 N_0^2$. If there are at least $\nu N_0$ color-balanced vertices in $V_1$, then there are at least $\nu N_0\times \mu^2 N_0^2=\nu\mu^2 N_0^3$ edges built in $E_L$, which means $H_L$ has at least $\delta N_0^3$ edges, and we would be done. 

It remains to consider the case in which there are fewer than $\nu N_0$ color-balanced vertices in $V_1$. 

The vertices of $V_1$ which are not color-balanced must have more than $(1-\mu) N_0$ neighbors of a single color. Let $S_R$ be the set of vertices with at least $(1-\mu) N_0$ red neighbors, and $S_B$ be the set of vertices with this many blue neighbors. We know $|S_R| + |S_B|>(1-\nu)N_0$. 

Since $G$ is $\varepsilon$-color-balanced, $d_R (G) \le 1-\varepsilon$. Also, by counting red edges from $S_R$ alone, $d_R(G) \geq |S_R|\cdot (1-\mu) / N_0$, so it follows that 
\[
|S_R|\le\frac{(1-\varepsilon)N_0}{1-\mu}=\frac{N_0}{1+\varepsilon}.
\]

This implies $|S_B|>(1-\nu)N_0-\frac{N_0}{1+\varepsilon}=\frac{\varepsilon N_0}{2(1+\varepsilon)}$. Likewise, we can show that $|S_R|>\frac{\varepsilon N_0}{2(1+\varepsilon)}$.

For each pair $(u_1, u_2) \in S_R \times S_B$, there must be at least $(1-2\mu)N_0$ vertices $v$ for which $(u_1, v)$ is red and $(u_2, v)$ is blue. Each triple $(u_1, u_2, v)$ gives an edge in the right vertex-pair incidence graph $H_R$. We find that the total number of edges in $H_R$ must be at least
\[
|S_R| \cdot |S_B| \cdot (1-2\mu) N_0 \ge \delta N_0^3.
\]

Thus, either $H_L$ or $H_R$ has at least $\delta N_0^3$ edges, as desired.
\end{proof}
\end{lemma}

We can now complete the proof of Theorem~\ref{thm:nlogn}. 

\noindent {\it Proof of Theorem~\ref{thm:nlogn}.} Builder's strategy is to first construct a complete multipartite graph $G=K_{(C\times N_0)}$, where $N_0 = N/C$. By Lemma~\ref{thm:multipartite}, if the $\varepsilon$-reduced graph of $G$ is complete, then $s(G) > \Omega(N^2) > \Omega(N\log N)$, as desired.

Otherwise, we can find an induced subgraph $H=G[U_i \cup U_j]$ on two of the parts of $G$ which is $\varepsilon$-color-balanced. By Lemma~\ref{lem:leastdensity}, it follows that (without loss of generality) the left vertex-pair incidence graph $H_L$ of $H$ has  at least $\delta N_0^3=\delta(\varepsilon)N_0^3$ edges. Then, by Lemma \ref{lem:computation}, $H_L$ has an induced subgraph $H^*$ isomorphic to $K_{a,b}$, where $a=\delta \log N_0$ and $b=N_0\log N_0$.

Let $P$ be the set of all pairs $(u,u')$ where $u$ is one of the $a$ vertices on the left side of $H^*$ and $u'$ is a vertex on the left side of $H$. We have that $|P| \ge \delta N_0 \log N_0$. Let $Q$ be the set of all pairs $(v_i, v_j)\in H^2$ represented by vertices on the right side of $H^*$, so that $|Q| = N_0 \log N_0$. Since $(u, v_i, v_j)$ is an edge of the vertex-pair incidence graph for every such $u, u', v_i ,v_j$, it follows that every $p \in P$ is independent from every $q\in Q$.

By Lemma~\ref{lem:pairwise}, $s(G) \ge \min(|P|, |Q|) \ge \Omega(N\log N)$, as desired.
\hfill \qed

\section{Closing Remarks}\label{sec:closing}

The off-diagonal case of the restricted online Ramsey game is equally interesting, and we believe even larger savings can be made here.

\begin{conj}\label{conj:off-diagonal}
There exists an absolute constant $c$ such that if $N=r(3,n)$, then
\[
\tilde{r}(3,n;N) \le (1-c)\binom{N}{2}.
\]
\end{conj}

Suppose Builder orders the vertices $v_1,\ldots, v_n$ arbitrarily and employs the following strategy. On step $i$, Builder builds all unbuilt pairs out of $v_i$ so far. However, during the course of the game, Builder will come across many edges $(v_i,v_j)$ with a common neighbor $v_k$ such that $(v_i,v_k)$ and $(v_j,v_k)$ are both red. In this case, if $(v_i,v_j)$ is built and colored red, then Builder obtains a red triangle and wins. 

We call such edges $(v_i, v_j)$ {\it forced edges}, edges that Painter will certainly paint blue, and Builder may skip building them until they can be used to fill in a complete blue $n$-clique with certainty. We conjecture that regardless of Painter's actions, either Builder will quickly obtain a blue $n$-clique, or else a constant fraction of the edges in $K_N$ will become forced. If true, this would prove Conjecture~\ref{conj:off-diagonal}.

We remark that if exponential improvements are made to the upper bounds on either the diagonal or off-diagonal restricted online Ramsey numbers, then such improvements would translate to exponential improvements on the unrestricted online Ramsey numbers as well. However, it seems unlikely that such improvements are even true.

Conlon, Fox, Grinshpun, and He \cite{xiaoyu} asked a somewhat different question about the restricted online Ramsey number. Fix $m,n \ge 3$ and letting $N$ vary, how does the quantity $\tilde{r}(m,n;N)$ change? In this paper we studied the diagonal case (where $m=n$) and let $N = r(m,n)$, the minimum value for which $\tilde{r}(m,n;N)$ is defined, while for $N$ sufficiently large, $\tilde{r}(m,n;N) = \tilde{r}(m,n)$. They conjectured that $\tilde{r}(m,n;N)$ decreases substantially as $N$ varies between these values. 
Even the simplest question, whether $\tilde{r}(m,n;N)>\tilde{r}(m,n)$ holds for any $N$, is unknown to us at this time.

\section{Acknowledgements}
We would like to thank George Schaeffer for organizing the Stanford Undergraduate Research in Mathematics program where this research took place, and Stanford University for the opportunity and funding to pursue this project. We are grateful to Jacob Fox for his valuable input and suggestions regarding the restricted online Ramsey game.

\end{document}